\documentclass[oneside,final,11pt]{amsart}

\usepackage{dsfont}
\usepackage[all]{xy}

\newcommand{\R}{\mathds R}
\newcommand{\cupdot}{\mathbin{\hspace{2pt}\cdot\hspace{-5pt}\cup}}

\title{Small Valdivia compacta and trees}

\author{Claudia Correa}
\thanks{The first author is sponsored by FAPESP (Process no.\ 2014/00848-2).}
\address{Departamento de Matem\'atica,\hfill\break\indent Universidade de S\~ao Paulo, Brazil}
\email{claudiac.mat@gmail.com}
\author{Daniel V. Tausk}
\address{Departamento de Matem\'atica,\hfill\break\indent Universidade de S\~ao Paulo, Brazil}
\email{tausk@ime.usp.br} \urladdr{http://www.ime.usp.br/\~{}tausk}

\date{July 23rd, 2016}

\begin{document}

\theoremstyle{plain}\newtheorem{teo}{Theorem}[section]
\theoremstyle{plain}\newtheorem{prop}[teo]{Proposition}
\theoremstyle{plain}\newtheorem{lem}[teo]{Lemma}
\theoremstyle{plain}\newtheorem{cor}[teo]{Corollary}
\theoremstyle{definition}\newtheorem{defin}[teo]{Definition}
\theoremstyle{remark}\newtheorem{rem}[teo]{Remark}
\theoremstyle{plain} \newtheorem{assum}[teo]{Assumption}
\theoremstyle{definition}\newtheorem{example}[teo]{Example}
\theoremstyle{plain}\newtheorem*{conjecture}{Conjecture}

\begin{abstract}
We present a characterization of Valdivia compact spaces of small weight in terms of path spaces of trees and we use it to obtain
(under $\diamondsuit$) a counterexample to a conjecture related to an open problem concerning twisted sums of $C(K)$ spaces.
\end{abstract}

\maketitle

\begin{section}{Introduction}
The purpose of this article is twofold: first, we use the description given by Kubi\'s and Michalewski in \cite{KubisSmall} of the class of small Valdivia compacta involving inverse limits of compact metric spaces to obtain a new characterization of this class in terms of path spaces of certain types of trees, endowed with appropriate compatible topologies. This characterization allows one to fine-tune the structure of a Valdivia compactum by manipulating the properties of the corresponding tree. We then use this technique to construct, under $\diamondsuit$, a counterexample to the following
conjecture stated in \cite[Section~4]{ExtCKc0}.
\begin{conjecture}
If $K$ is a nonempty Valdivia compact space satisfying the countable chain condition (ccc), then either $K$ has a $G_\delta$ point or $K$ admits a nontrivial convergent sequence in the complement of a dense $\Sigma$-subset.
\end{conjecture}
Recall that a compact Hausdorff space $K$ is called a {\em Valdivia compactum\/} if it admits a dense {\em $\Sigma$-subset}, i.e., a subset of $K$ of the form
$\varphi^{-1}[\Sigma(\Gamma)]$, where $\varphi:K\to\R^\Gamma$ is a continuous injection and $\Sigma(\Gamma)$ denotes the set of points $x\in\R^\Gamma$ such that $\big\{\gamma\in\Gamma:x_\gamma\ne0\big\}$ is countable; here $\Gamma$ is an arbitrary index set and $\R^\Gamma$ is endowed with the product topology.
We call a Valdivia compactum {\em small\/} if its weight is not greater than $\omega_1$.
Valdivia compact spaces constitute a large superclass of Corson compact spaces, closed under arbitrary products, and they were introduced by Argyros, Mercourakis, and Negrepontis in \cite{ArgyrosCorson}. This class and its relation to the theory of Banach spaces have since then been studied by several authors \cite{ArgyrosLondon, Kalenda3, Kalenda2, KK2, KubisSmall, Valdivia, Valdivia2} (see \cite{Kalenda} for a survey).

We observe that finding examples of nonempty Valdivia compact spaces with no $G_\delta$ points and no nontrivial convergent sequences in the complement of a dense $\Sigma$-subset is not a trivial task, since the absence of $G_\delta$ points tends to make the complement of dense $\Sigma$-subsets ``large''
(see, for instance, \cite[Theorem~3.3]{Kalenda} for a more precise statement). In \cite[Proposition~4.7]{ExtCKc0}, it was shown that
the path space of a certain tree $T$, endowed with the product topology of $2^T$, provides such an example. However, using this topology it is not possible to have a nonempty path space with no $G_\delta$ points and ccc. The techniques presented in this article allow us to handle more complicated topologies on path spaces and after a technically elaborate construction a counterexample to the Conjecture is obtained under $\diamondsuit$. Recall that $\diamondsuit$ is a combinatorial
principle stronger than the {\it continuum hypothesis\/} (CH) and consistent with ZFC (see \cite{Kunen}).

\smallskip

As shown in \cite{ExtCKc0}, the Conjecture implies, under CH, that for every nonmetrizable Valdivia
compact space $K$ there exists a {\em nontrivial twisted sum\/} of $c_0$ and $C(K)$, i.e., a Banach space $X$ containing a noncomplemented isomorphic copy
$Y$ of $c_0$ such that $X/Y$ is isomorphic to $C(K)$. As usual, $C(K)$ denotes the Banach space of continuous real-valued functions on $K$ endowed with
the supremum norm. The existence of a nontrivial twisted sum of $c_0$ and $C(K)$, for every
nonmetrizable compact Hausdorff space $K$, is an open problem discussed in many articles \cite{CastilloKalton,JesusSobczyk,Castilloscattered}.
This problem remains open even in the context of Valdivia compact spaces (\cite{Castilloscattered}) and was only recently settled, under CH, for Corson compacta (\cite[Theorem~3.1]{ExtCKc0}). In fact, the existence of a nontrivial twisted sum of $c_0$ and $C(K)$ for every nonmetrizable Corson compact space $K$ is known to hold also under Martin's Axiom (\cite[Remark~3.5]{ExtCKc0}), but it is not known whether this can be proven in ZFC.

Even though the Conjecture turned out to be false (under $\diamondsuit$), a deeper understanding of the class
of counterexamples to the Conjecture should shed light upon the problem of existence of nontrivial twisted sums of $c_0$ and $C(K)$, for $K$ a nonmetrizable Valdivia compact space. The authors do not know whether such a nontrivial twisted sum exists if $K$ is the counterexample to the conjecture constructed in this article. We note, in addition, that to prove the existence of a nontrivial twisted sum of $c_0$ and $C(K)$ for an arbitrary nonmetrizable Valdivia compact space $K$, one can restrict attention to the case when $K$ is small, since every nonmetrizable Valdivia compact space contains a nonmetrizable small Valdivia subspace as a retract (\cite[Lemma~1.3]{ArgyrosCorson}).

\smallskip

Here is an overview of the contents of this article. In Section~\ref{sec:inverseValdivia}, we review the relevant material concerning inverse limits
and we recall the characterization of small Valdivia compacta from \cite{KubisSmall}. In Section~\ref{sec:inversetree}, we develop the theory relating
inverse limits to trees endowed with some additional structure and we study the relevant topologies on the path space.
Finally, Section~\ref{sec:counterexample} is devoted to the construction of the counterexample to the Conjecture.

\newpage

\end{section}

\begin{section}{Inverse limits and Valdivia compacta}\label{sec:inverseValdivia}

Throughout the article, we denote by $\vert\mathcal X\vert$ the cardinality of a set $\mathcal X$, by $w(\mathcal X)$ the weight of a topological
space $\mathcal X$ and by $\mathbb S$ the class of all successor ordinals. We start by recalling some standard definitions and facts concerning inverse limits.
\begin{defin}\label{thm:definversesystem}
Let $(I,{\le})$ be a partially ordered directed set. An {\em inverse system of sets\/} $\mathcal K=\big((K_i)_{i\in I},(r_{ij})_{i\le j\in I}\big)$ indexed on $I$ consists of a family $(K_i)_{i\in I}$ of sets and a family $(r_{ij}:K_j\to K_i)_{i\le j\in I}$ of maps such that
$r_{ii}$ is the identity of $K_i$, for all $i\in I$, and $r_{ij}\circ r_{jk}=r_{ik}$, for all
$i,j,k\in I$ with $i\le j\le k$. We call $\big(K,(r_i)_{i\in I}\big)$ a {\em cone\/} over $\mathcal K$ if $K$ is a set and $(r_i:K\to K_i)_{i\in I}$ is a family of maps with $r_{ij}\circ r_j=r_i$, for all $i,j\in I$ with $i\le j$. An {\em inverse limit\/} of $\mathcal K$ is a cone
$\big(K,(r_i)_{i\in I}\big)$ over $\mathcal K$ such that, for every cone $\big(K',(r'_i)_{i\in I}\big)$ over $\mathcal K$,
there exists a unique map $f:K'\to K$ such that $r_i\circ f=r'_i$, for all $i\in I$.
\end{defin}
A concrete description of an inverse limit of $\mathcal K$ is obtained by considering the set:
\begin{equation}\label{eq:canliminv}
\Big\{(x_i)_{i\in I}\in\prod_{i\in I}K_i:\text{$r_{ij}(x_j)=x_i$, for all $i,j\in I$ with $i\le j$}\Big\},
\end{equation}
together with the restrictions of the projections. Note that $\big(K,(r_i)_{i\in I}\big)$ is a cone over $\mathcal K$ if and only if the image of the map $(r_i)_{i\in I}:K\to\prod_{i\in I}K_i$ is contained in \eqref{eq:canliminv} and that $\big(K,(r_i)_{i\in I}\big)$ is an inverse limit of $\mathcal K$ if and only if the map $(r_i)_{i\in I}$ is a bijection between $K$ and \eqref{eq:canliminv}.
When the sets $K_i$ are endowed with compact Hausdorff topologies such that the maps $r_{ij}$ are continuous, we call $\mathcal K$ an {\em inverse system of
compact Hausdorff spaces}. Cones and inverse limits are then defined by replacing ``set'' with ``compact Hausdorff space'' and ``map'' with ``continuous map'' in Definition~\ref{thm:definversesystem}. In this context, the set \eqref{eq:canliminv} should be endowed with the product topology.

\smallskip

The following simple results give information on the closed $G_\delta$ subsets of an inverse limit.
\begin{lem}\label{thm:lemasaturado}
Let $\mathcal K=\big((K_i)_{i\in I},(r_{ij})_{i\le j\in I}\big)$ be an inverse system of compact Hausdorff spaces and $\big(K,(r_i)_{i\in I}\big)$ be an
inverse limit of $\mathcal K$. Given a closed subset $F$ of $K$ and an open subset $U$ of $K$ containing $F$, there exists $i\in I$
with $r_i^{-1}\big[r_i[F]\big]\subset U$. In particular, if every countable subset of $I$ is bounded and $F$ is a closed $G_\delta$ subset of $K$,
then there exists $i\in I$ with $F=r_i^{-1}\big[r_i[F]\big]$.
\end{lem}
\begin{proof}
Note that $\big\{(x,y)\in K\times K:r_i(x)=r_i(y)\big\}$, $i\in I$, is a downward directed family of closed subsets of the compact space $K\times K$
whose intersection is the diagonal. This intersection is contained in the complement of the closed set $F\times(K\setminus U)$ and therefore
$\big\{(x,y)\in K\times K:r_i(x)=r_i(y)\big\}$ is contained in the complement of $F\times(K\setminus U)$, for some $i\in I$.
\end{proof}

\begin{cor}\label{thm:KGdeltapoint}
Let $\mathcal K=\big((K_i)_{i\in I},(r_{ij})_{i\le j\in I}\big)$ be an inverse system of compact metric spaces and assume that every countable subset of $I$ is bounded. If $\big(K,(r_i)_{i\in I}\big)$ is an inverse limit of $\mathcal K$, then the set:
\[\big\{r_i^{-1}[F]:\text{$i\in I$, $F$ closed in $K_i$}\big\}\]
coincides with the collection of closed $G_\delta$ subsets of $K$. In particular, $K$ has a $G_\delta$ point if and only if there exist $i\in I$ and
$x\in K_i$ with $\vert r_i^{-1}(x)\vert=1$.\qed
\end{cor}

\begin{defin}
Let $\mathcal K=\big((K_i)_{i\in I},(r_{ij})_{i\le j\in I}\big)$ be an inverse system of sets (resp., of compact Hausdorff spaces).
A {\em right inverse\/} of $\mathcal K$ is a family of maps (resp., of continuous maps) $(\sigma_{ij}:K_i\to K_j)_{i\le j\in I}$ such that $\sigma_{ij}$ is a right inverse of $r_{ij}$ and $\sigma_{jk}\circ\sigma_{ij}=\sigma_{ik}$, for all $i,j,k\in I$ with $i\le j\le k$.
\end{defin}
If $(\sigma_{ij})_{i\le j\in I}$ is a right inverse of $\mathcal K$ and $\big(K,(r_i)_{i\in I}\big)$ is an inverse limit of $\mathcal K$, then there exists a unique family of maps $(\sigma_i:K_i\to K)_{i\in I}$ such that $\sigma_i$ is a right inverse of $r_i$ and $\sigma_j\circ\sigma_{ij}=\sigma_i$ for all $i,j\in I$ with $i\le j$; these maps are automatically continuous if $\mathcal K$ is an inverse system of compact Hausdorff spaces (\cite[Lemma~3.1]{KubisSmall}). We call the maps $\sigma_i$ {\em induced\/} by the right inverse $(\sigma_{ij})_{i\le j\in I}$ and they will always be denoted by $\sigma_i$, whenever $(\sigma_{ij})_{i\le j\in I}$ denotes a right inverse of an inverse system. Note that, for $i,j\in I$, we have $r_j\circ\sigma_i=\sigma_{ij}$, if $i\le j$, and $r_j\circ\sigma_i=r_{ji}$, if $j\le i$.


\begin{defin}
An inverse system $\mathcal K=\big((K_\alpha)_{\alpha\in\theta},(r_{\alpha\beta})_{\alpha\le\beta\in\theta}\big)$ whose index set $\theta$ is
a nonzero ordinal (endowed with the natural order) is called an {\em inverse $\theta$-sequence\/} or, more simply, an {\em inverse sequence}. This inverse sequence is called {\em continuous\/} if,
for every limit ordinal $\alpha\in\theta$, we have that $\big(K_\alpha,(r_{\beta\alpha})_{\beta\in\alpha}\big)$ is an inverse limit of
$\big((K_\beta)_{\beta\in\alpha},(r_{\beta\gamma})_{\beta\le\gamma\in\alpha}\big)$.
\end{defin}

The next lemma gives a useful description of the image of the maps $\sigma_\alpha$ induced by a right inverse of a continuous inverse sequence.
\begin{lem}\label{thm:imagesigmaalpha}
If $\mathcal K=\big((K_\alpha)_{\alpha\in\theta},(r_{\alpha\beta})_{\alpha\le\beta\in\theta}\big)$ is a continuous inverse sequence with
a right inverse $(\sigma_{\alpha\beta})_{\alpha\le\beta\in\theta}$ and $\big(K,(r_\alpha)_{\alpha\in\theta}\big)$ is an inverse limit of $\mathcal K$, then for all $\alpha\in\theta$, the image of $\sigma_\alpha$ is equal to:
\begin{equation}\label{eq:imsigmaalpha}
\big\{x\in K:\text{$r_{\beta+1}(x)\in\sigma_{\beta,\beta+1}[K_\beta]$, for all $\beta\ge\alpha$ with $\beta+1<\theta$}\big\}.
\end{equation}
\end{lem}
\begin{proof}
Obviously $\sigma_\alpha[K_\alpha]$ is contained in \eqref{eq:imsigmaalpha}. To prove the equality, take $x$ in \eqref{eq:imsigmaalpha} and show by induction on $\beta$ that $r_\beta\big((\sigma_\alpha\circ r_\alpha)(x)\big)=r_\beta(x)$, for all $\beta\ge\alpha$. Since the maps $r_\beta$, $\beta\ge\alpha$, separate the points of $K$, it follows that $(\sigma_\alpha\circ r_\alpha)(x)=x$.
\end{proof}

A collection $\mathcal D$ of nonempty open subsets of a topological space $\mathcal X$ is called {\em dense\/} in the topology of $\mathcal X$ if every nonempty open subset of $\mathcal X$ contains an element of $\mathcal D$. Obviously, if $\mathcal D$ is dense, then $\mathcal X$ has ccc if and only if every family of pairwise disjoint elements
of $\mathcal D$ is countable. We are interested in determining conditions under which an inverse limit has ccc and to this aim we describe in the next lemma a convenient dense subset of its topology.
\begin{lem}
Let $\mathcal K=\big((K_\alpha)_{\alpha\in\theta},(r_{\alpha\beta})_{\alpha\le\beta\in\theta}\big)$ be a continuous inverse sequence of compact Hausdorff spaces with a right inverse $(\sigma_{\alpha\beta})_{\alpha\le\beta\in\theta}$ and let $\big(K,(r_\alpha)_{\alpha\in\theta}\big)$ be an inverse limit of $\mathcal K$. Consider the collection $\mathcal D$ of open subsets of $K$ of the form $r_{\alpha+1}^{-1}[U]$, with $U$ a nonempty open subset
of $K_{\alpha+1}\setminus\sigma_{\alpha,\alpha+1}[K_\alpha]$ and $\alpha+1<\theta$. If $\bigcup_{\alpha\in\theta}\sigma_\alpha[K_\alpha]$ has empty interior in $K$, then $\mathcal D$ is dense in the topology of $K$.
\end{lem}
\begin{proof}
Let $V$ be a nonempty open subset of $K$ and pick $x\in V$ not in $\bigcup_{\alpha\in\theta}\sigma_\alpha[K_\alpha]$. By Lemma~\ref{thm:lemasaturado}, there exists $\beta\in\theta$ with $r_\beta^{-1}\big(r_\beta(x)\big)\subset V$ and by Lemma~\ref{thm:imagesigmaalpha} there exists $\alpha\ge\beta$ with
$\alpha+1<\theta$ and $r_{\alpha+1}(x)$ not in $\sigma_{\alpha,\alpha+1}[K_\alpha]$. Then $r_{\alpha+1}^{-1}\big(r_{\alpha+1}(x)\big)\subset r_\beta^{-1}\big(r_\beta(x)\big)\subset V$ and setting
\[U=K_{\alpha+1}\setminus\big(r_{\alpha+1}[K\setminus V]\cup\sigma_{\alpha,\alpha+1}[K_\alpha]\big),\]
we have that $r_{\alpha+1}(x)\in U$ and $r_{\alpha+1}^{-1}[U]\subset V$.
\end{proof}

\begin{cor}\label{thm:Kccc}
Let $\mathcal K=\big((K_\alpha)_{\alpha\in\omega_1},(r_{\alpha\beta})_{\alpha\le\beta\in\omega_1}\big)$ be a continuous inverse sequence of compact metric spaces with a right inverse $(\sigma_{\alpha\beta})_{\alpha\le\beta\in\omega_1}$ and let $\big(K,(r_\alpha)_{\alpha\in\omega_1}\big)$ be an inverse limit
of $\mathcal K$. Assume that $\bigcup_{\alpha\in\omega_1}\sigma_\alpha[K_\alpha]$ has empty interior
in $K$. Then $K$ does not have ccc if and only if there exist an uncountable subset $\Lambda$ of $\omega_1$ and a family
$(U_\alpha)_{\alpha\in\Lambda}$, with each $U_\alpha$ a nonempty open subset of $K_{\alpha+1}\setminus\sigma_{\alpha,\alpha+1}[K_\alpha]$,
such that the sets $r_{\alpha+1}^{-1}[U_\alpha]$, $\alpha\in\Lambda$, are pairwise disjoint.
\end{cor}
\begin{proof}
Follows directly from the lemma, keeping in mind that the spaces $K_\alpha$ have ccc.
\end{proof}

We recall the following characterization of small Valdivia compacta in terms of inverse limits.
\begin{teo}\label{thm:ValdiviaKubis}
Let $K$ be a Valdivia compact space with $w(K)\le\omega_1$ and let $S$ be a dense $\Sigma$-subset of $K$. Then, there exist a continuous
inverse sequence $\mathcal K=\big((K_\alpha)_{\alpha\in\omega_1},(r_{\alpha\beta})_{\alpha\le\beta\in\omega_1}\big)$ of compact metric spaces with a right inverse $(\sigma_{\alpha\beta})_{\alpha\le\beta\in\omega_1}$ and a family of continuous maps $(r_\alpha:K\to K_\alpha)_{\alpha\in\omega_1}$ such that $\big(K,(r_\alpha)_{\alpha\in\omega_1}\big)$ is an inverse limit of $\mathcal K$ and
\begin{equation}\label{eq:Sigmasubset}
S=\bigcup_{\alpha\in\omega_1}\sigma_\alpha[K_\alpha].
\end{equation}
Conversely, if $\big(K,(r_\alpha)_{\alpha\in\omega_1}\big)$ is an inverse limit of a continuous inverse sequence $\mathcal K=\big((K_\alpha)_{\alpha\in\omega_1},(r_{\alpha\beta})_{\alpha\le\beta\in\omega_1}\big)$ of compact metric spaces with a right inverse $(\sigma_{\alpha\beta})_{\alpha\le\beta\in\omega_1}$, then $K$ is a Valdivia compact space with $w(K)\le\omega_1$ and \eqref{eq:Sigmasubset} is a dense $\Sigma$-subset of $K$.
\end{teo}
\begin{proof}
The first part of the statement follows from \cite[Lemma~1.3]{ArgyrosCorson} and the fact that \eqref{eq:Sigmasubset} is a dense $\Sigma$-subset of $K$ follows from the proof of \cite[Theorem~4.2]{KubisSmall}.
\end{proof}

In view of Theorem~\ref{thm:ValdiviaKubis}, when a small Valdivia compact space $K$ is represented as an inverse limit of a continuous $\omega_1$-sequence
of compact metric spaces with a right inverse, the assumption of Corollary~\ref{thm:Kccc} states that a certain $\Sigma$-subset of $K$ has empty interior.
It turns out that this condition is satisfied for a Valdivia compact space without $G_\delta$ points, as we now show.
Recall that a compact Hausdorff space is said to be {\em Corson\/} if it is a $\Sigma$-subset of itself.
\begin{lem}\label{thm:ValdiviaGdeltaSigma}
If $K$ is a Valdivia compact space without $G_\delta$ points, then every $\Sigma$-subset of $K$ has empty interior.
\end{lem}
\begin{proof}
If $S$ is a $\Sigma$-subset of $K$ with nonempty interior, then $S$ contains a nonempty closed $G_\delta$ subset $F$ of $K$. It follows that $F$ is a nonempty Corson compact space and therefore has a $G_\delta$ point (\cite[Theorem~3.3]{Kalenda}).
\end{proof}

\end{section}

\begin{section}{Inverse limits and trees}\label{sec:inversetree}

Recall that a {\em tree\/} is a partially ordered set $(T,{\le})$ such that, for all $t\in T$, the set $\left]\cdot,t\right[=\big\{s\in T:s<t\big\}$ is well-ordered. A subset $X$ of $T$ is called an {\em initial part\/} of $T$ if $\left]\cdot,t\right[\subset X$, for all $t\in X$; a {\em chain\/} if it is totally ordered; an {\em antichain\/} if any two distinct elements of $X$ are incomparable; a {\em path\/} if it is both a chain and an initial part of $T$.
We say that $X\subset T$ satisfies the {\em countable chain condition\/} (ccc) if every antichain contained in $X$ is countable. We denote by $P(T)$ the set of all paths of $T$ and by $P^*(T)$ the set of nonempty paths of $T$. We have a canonical embedding $\mathfrak p$ of $(T,{\le})$ into $\big(P^*(T),{\subset}\big)$ defined by:
\[\mathfrak p(t)=\big\{s\in T:s\le t\big\},\]
for all $t\in T$. Any subset of $T$ endowed with the restriction of $\le$ is itself a tree and is called a {\em subtree\/}
of $T$; an initial part of $T$ endowed with the restriction of $\le$ is called an {\em initial subtree\/} of $T$. If $Z$ is an initial subtree of $T$,
then the paths of $Z$ are precisely the paths of $T$ that are contained in $Z$. Given a path $A\in P(T)$, we set:
\[N_A=\big\{t\in T:\left]\cdot,t\right[=A\big\}.\]

We now introduce the structure that makes the connection between trees and inverse systems.
\begin{defin}
A {\em graded tree\/} $(T,\delta)$ consists of a tree $T$ and a mapping $\delta:T\to\mathbb S\cup\{0\}$ such that $\delta(t)=0$, for every minimal element $t$ of $T$, and $\delta(t)<\delta(s)$, for all $t,s\in T$ with $t<s$. The map $\delta$ is called a {\em grading function\/} for $T$. For every ordinal $\alpha$,
we denote by $T_\alpha$ the initial subtree of $T$ defined as
\[T_\alpha=\big\{t\in T:\delta(t)\le\alpha\big\}\]
and by $\rho_\alpha:P^*(T)\to P^*(T_\alpha)$ the map given by $\rho_\alpha(A)=A\cap T_\alpha$, for all $A\in P^*(T)$.
By a {\em compatible topology\/} on $P^*(T)$ we mean a compact Hausdorff topology on $P^*(T)$ such that the maps $\rho_\alpha:P^*(T)\to P^*(T)$ are continuous, for every ordinal $\alpha$. Given a nonzero ordinal $\theta$, we say that the graded tree $(T,\delta)$ is {\em $\theta$-graded\/} if $\delta(t)<\theta$, for all $t\in T$.
\end{defin}
Note that if $(T,\delta)$ is a graded tree, then:
\begin{equation}\label{eq:zero}
P^*(T_0)=\mathfrak p[\delta^{-1}(0)]
\end{equation}
and, for every ordinal $\lambda$,
\begin{equation}\label{eq:passo}
P^*(T_{\lambda+1})=P^*(T_\lambda)\cupdot\mathfrak p[\delta^{-1}(\lambda+1)],
\end{equation}
with the union in \eqref{eq:passo} being disjoint.

\smallskip

A $\theta$-graded tree $(T,\delta)$ is associated in a natural way with a continuous inverse $\theta$-sequence of sets $\mathcal K_\theta(T,\delta)$ with a right inverse (Proposition~\ref{thm:KthetaTdelta}); moreover, a compatible topology on $P^*(T)$ makes $\mathcal K_\theta(T,\delta)$ an inverse sequence of compact Hausdorff spaces. It turns
out (Proposition~\ref{thm:converseKthetaTdelta}) that every continuous inverse $\theta$-sequence with a right inverse is of the form $\mathcal K_\theta(T,\delta)$, up to isomorphism.
\begin{prop}\label{thm:KthetaTdelta}
Let $(T,\delta)$ be a $\theta$-graded tree and set:
\[\mathcal K_\theta(T,\delta)=\Big(\big(P^*(T_\alpha)\big)_{\alpha\in\theta},(\rho_{\alpha\beta})_{\alpha\le\beta\in\theta}\Big),\]
where $\rho_{\alpha\beta}=\rho_\alpha\vert_{P^*(T_\beta)}:P^*(T_\beta)\to P^*(T_\alpha)$. Then:
\begin{itemize}
\item[(a)] $\mathcal K_\theta(T,\delta)$ is a continuous inverse sequence of sets;
\item[(b)] $\big(P^*(T),(\rho_\alpha)_{\alpha\in\theta}\big)$ is an inverse limit of $\mathcal K_\theta(T,\delta)$;
\item[(c)] the inclusion maps $P^*(T_\alpha)\to P^*(T_\beta)$, $\alpha\le\beta\in\theta$, constitute a right inverse of $\mathcal K_\theta(T,\delta)$
and the induced maps $P^*(T_\alpha)\to P^*(T)$ are the inclusion maps as well.
\end{itemize}
Moreover, if $P^*(T)$ is endowed with a compatible topology and each $P^*(T_\alpha)$ has the subspace topology, then (a), (b) and (c) hold when ``set'' is replaced with ``compact Hausdorff space''.
\end{prop}
\begin{proof}
The proof of (a), (b) and (c) is straightforward. For the last statement, note that $P^*(T_\alpha)$ is closed in $P^*(T)$, being the set of fixed
points of the continuous map $\rho_\alpha$.
\end{proof}

\begin{prop}\label{thm:converseKthetaTdelta}
If $\mathcal K=\big((K_\alpha)_{\alpha\in\theta},(r_{\alpha\beta})_{\alpha\le\beta\in\theta}\big)$ is a continuous inverse sequence of sets and $(\sigma_{\alpha\beta})_{\alpha\le\beta\in\theta}$ is a right inverse of $\mathcal K$, then there exist a $\theta$-graded tree $(T,\delta)$ and
a family of bijections $\varphi_\alpha:K_\alpha\to P^*(T_\alpha)$ such that, for all $\alpha,\beta\in\theta$ with
$\alpha\le\beta$, the diagrams:
\begin{equation}\label{eq:diagrams}
\vcenter{\xymatrix@C+5pt{%
K_\beta\ar[r]^-{\varphi_\beta}\ar[d]_{r_{\alpha\beta}}&P^*(T_\beta)\ar[d]^{\rho_{\alpha\beta}}\\
K_\alpha\ar[r]^-{\varphi_\alpha}&P^*(T_\alpha)}}\qquad
\vcenter{\xymatrix@C+5pt{%
K_\beta\ar[r]^-{\varphi_\beta}&P^*(T_\beta)\\
K_\alpha\ar[r]^-{\varphi_\alpha}\ar[u]^{\sigma_{\alpha\beta}}&P^*(T_\alpha)\ar[u]_{\text{inclusion}}}}
\end{equation}
commute. In particular, if $\big(K,(r_\alpha)_{\alpha\in\theta}\big)$ is an inverse limit of $\mathcal K$, then there exists a bijection
$\varphi:K\to P^*(T)$ such that the diagrams:
\[\xymatrix@C+5pt{%
K\ar[r]^-{\varphi}\ar[d]_{r_\alpha}&P^*(T)\ar[d]^{\rho_\alpha}\\
K_\alpha\ar[r]^-{\varphi_\alpha}&P^*(T_\alpha)}\qquad
\xymatrix@C+5pt{%
K\ar[r]^-{\varphi}&P^*(T)\\
K_\alpha\ar[r]^-{\varphi_\alpha}\ar[u]^{\sigma_\alpha}&P^*(T_\alpha)\ar[u]_{\text{inclusion}}}\]
commute, for all $\alpha\in\theta$. Moreover, if $\mathcal K$ is a continuous inverse sequence of compact Hausdorff spaces, then the topology
on $P^*(T)$ that makes $\varphi$ a homeomorphism is a compatible topology and each map $\varphi_\alpha$ is a homeomorphism if $P^*(T_\alpha)$
is endowed with the subspace topology.
\end{prop}
\begin{proof}
Consider the set $\overline T=\bigcup_{\alpha\in\theta}\big(K_\alpha\times\{\alpha\}\big)$ endowed with the partial order defined by:
\[(x,\alpha)\le(y,\beta)\Longleftrightarrow\text{$\alpha\le\beta$ and $r_{\alpha\beta}(y)=x$}.\]
For every $t\in\overline T$, the projection $\pi_2:\overline T\to\theta$ restricts to an order isomorphism between $\left]\cdot,t\right[$ and a subset of
$\theta$, so that $\overline T$ is a tree. Let $T$ be the subtree of $\overline T$ defined by:
\[T=\big(K_0\times\{0\}\big)\cup\bigcup\big\{\big(K_{\alpha+1}\setminus\sigma_{\alpha,\alpha+1}[K_\alpha]\big)\times\{\alpha+1\}:\text{$\alpha\in\theta$
with $\alpha+1\in\theta$}\big\}\]
and let $\delta:T\to\mathbb S\cup\{0\}$ be the $\theta$-grading function given by the restriction of $\pi_2$.
For each $\alpha\in\theta$, define $\varphi_\alpha:K_\alpha\to P^*(T_\alpha)$ by setting:
\[\varphi_\alpha(x)=\big\{(z,\gamma)\in T:(z,\gamma)\le(x,\alpha)\big\}\]
for all $x\in K_\alpha$. It is clear that the left diagram in \eqref{eq:diagrams} commutes. To see that the right diagram commutes,
pick $x\in K_\alpha$ and observe that every $t\in\overline T$ with $(x,\alpha)<t\le\big(\sigma_{\alpha\beta}(x),\beta\big)$ is not in $T$; namely, if such $t$ were in $T$, it would be of the form $(z,\gamma+1)$, with $\alpha\le\gamma<\beta$ and:
\[z=r_{\gamma+1,\beta}\big(\sigma_{\alpha\beta}(x)\big)=\sigma_{\alpha,\gamma+1}(x)=\sigma_{\gamma,\gamma+1}\big(\sigma_{\alpha\gamma}(x)\big),\]
which contradicts $z\not\in\sigma_{\gamma,\gamma+1}[K_\gamma]$.
The fact that the maps $\varphi_\alpha$ are bijective is proven by induction on $\alpha$ using the commutativity of the diagrams in \eqref{eq:diagrams},
keeping in mind \eqref{eq:passo} and the continuity of the inverse sequences $\mathcal K$ and $\mathcal K_\theta(T,\delta)$. The proof of the remaining
assertions in the statement of the proposition is straightforward.
\end{proof}

Using the relation between graded trees and continuous inverse sequences with a right inverse presented above, we obtain now a characterization
of small Valdivia compacta in terms of graded trees.
\begin{teo}\label{thm:propValdiviatree}
Let $K$ be a Valdivia compact space with $w(K)\le\omega_1$ and let $S$ be a dense $\Sigma$-subset of $K$. Then, there exist an $\omega_1$-graded tree
$(T,\delta)$, a compatible topology on $P^*(T)$ such that $P^*(T_\alpha)$ is metrizable, for all $\alpha\in\omega_1$, and a homeomorphism
between $K$ and $P^*(T)$ that carries $S$ to the set:
\begin{equation}\label{eq:sigmasubsettree}
\bigcup_{\alpha\in\omega_1}P^*(T_\alpha)=\big\{A\in P^*(T):\vert A\vert\le\omega\big\}.
\end{equation}
Conversely, given an $\omega_1$-graded tree $(T,\delta)$ and a compatible topology on $P^*(T)$ such that $P^*(T_\alpha)$ is metrizable, for all
$\alpha\in\omega_1$, then $K=P^*(T)$ is a Valdivia compact space with $w(K)\le\omega_1$ and \eqref{eq:sigmasubsettree} is a dense $\Sigma$-subset of $K$.
\end{teo}
\begin{proof}
Follows from Theorem~\ref{thm:ValdiviaKubis} using Propositions~\ref{thm:KthetaTdelta} and \ref{thm:converseKthetaTdelta}.
\end{proof}

Let $(T,\delta)$ be a graded tree and let $P^*(T)$ be endowed with a compatible topology. It follows from \eqref{eq:passo} that, for every $\alpha\in\mathbb S$, the set $\mathfrak p[\delta^{-1}(\alpha)]$ is open in $P^*(T_\alpha)$. Thus, the topology on $\delta^{-1}(\alpha)$ induced by $\mathfrak p$
is locally compact Hausdorff and, by \eqref{eq:zero}, the topology on $\delta^{-1}(0)$ induced by $\mathfrak p$ is compact Hausdorff.
Our goal is to construct a compatible topology on $P^*(T)$ from given locally compact Hausdorff topologies on the sets $\delta^{-1}(\alpha)$ satisfying certain compatibility conditions. To this aim, we give the following definition.
\begin{defin}
Let $(T,\delta)$ be a graded tree. For each $\alpha\in\mathbb S\cup\{0\}$, we denote by $g_\alpha:P^*(T)\to\delta^{-1}(\alpha)\cupdot\{\infty\}$ the map defined by $g_\alpha(A)=t$, if $t$ is the (automatically unique) element of $A\cap\delta^{-1}(\alpha)$, and $g_\alpha(A)=\infty$, if $A\cap\delta^{-1}(\alpha)$ is empty; here $\infty$ denotes any point not in $\delta^{-1}(\alpha)$. Given $\alpha,\beta\in\mathbb S\cup\{0\}$ with $\alpha\le\beta$, we set
$g_{\alpha\beta}=g_\alpha\circ\mathfrak p\vert_{\delta^{-1}(\beta)}:\delta^{-1}(\beta)\to\delta^{-1}(\alpha)\cupdot\{\infty\}$.
We call the maps $g_{\alpha\beta}$ the {\em connecting maps\/} of the graded tree $(T,\delta)$.
\end{defin}

When a locally compact Hausdorff topology is given on a set $\mathcal X$, we always endow the disjoint union $\mathcal X\cupdot\{\infty\}$ with the unique compact Hausdorff topology that induces the given topology on $\mathcal X$, i.e., $\mathcal X\cupdot\{\infty\}$ is the one-point compactification of
$\mathcal X$, if $\mathcal X$ is not compact, and the point $\infty$ is isolated in $\mathcal X\cupdot\{\infty\}$, otherwise.
\begin{prop}\label{thm:maquininha}
Let $(T,\delta)$ be a graded tree. If $P^*(T)$ is endowed with a compatible topology and the sets $\delta^{-1}(\alpha)$, $\alpha\in\mathbb S\cup\{0\}$,
are endowed with the topologies induced by $\mathfrak p$, then the maps $g_\alpha$ and the connecting maps $g_{\alpha\beta}$ are continuous, for all $\alpha,\beta\in\mathbb S\cup\{0\}$ with $\alpha\le\beta$; moreover, the topology of $P^*(T)$ coincides with the topology induced by the maps $g_\alpha$, $\alpha\in\mathbb S\cup\{0\}$. Conversely, let the set $\delta^{-1}(\alpha)$ be endowed with a locally compact Hausdorff topology $\tau_\alpha$, for each $\alpha\in\mathbb S$, and let $\delta^{-1}(0)$ be endowed with a compact Hausdorff topology $\tau_0$. If the connecting maps of $(T,\delta)$ are continuous, then there exists a unique compatible topology on $P^*(T)$ such that $\tau_\alpha$ is equal to the topology on $\delta^{-1}(\alpha)$ induced by $\mathfrak p$, for all $\alpha\in\mathbb S\cup\{0\}$.
\end{prop}
\begin{proof}
Let $P^*(T)$ be endowed with a compatible topology and let each set $\delta^{-1}(\alpha)$ be endowed with the topology induced by $\mathfrak p$. To see that $g_\alpha$ is continuous, note that $g_\alpha=q\circ\rho_\alpha$, where $q:P^*(T_\alpha)\to\delta^{-1}(\alpha)\cupdot\{\infty\}$ is
defined by $q(A)=\mathfrak p^{-1}(A)$, for $A\in\mathfrak p[\delta^{-1}(\alpha)]$, and $q(A)=\infty$, otherwise. The continuity of the connecting maps then follows. The fact that the topology of $P^*(T)$ is induced by the maps $g_\alpha$ is a consequence of the observation that the map:
\begin{equation}\label{eq:todasgalpha}
(g_\alpha)_{\alpha\in\delta[T]}:P^*(T)\longrightarrow\prod_{\alpha\in\delta[T]}\big(\delta^{-1}(\alpha)\cupdot\{\infty\}\big)
\end{equation}
is continuous and injective. To prove the converse, let $P^*(T)$ be endowed with the topology induced by \eqref{eq:todasgalpha} and
let us show that this topology satisfies the required properties.
It is easy to see that the image of \eqref{eq:todasgalpha} is equal to the set:
\begin{multline*}
F=\Big\{(t_\alpha)_{\alpha\in\delta[T]}\in\prod_{\alpha\in\delta[T]}\big(\delta^{-1}(\alpha)\cupdot\{\infty\}\big):\text{$t_0\ne\infty$ and, for all $\alpha,\beta\in\delta[T]$}\\
\text{with $\alpha\le\beta$, if $t_\beta\ne\infty$, then $g_{\alpha\beta}(t_\beta)=t_\alpha$}\Big\}
\end{multline*}
and that the continuity of the connecting maps implies that $F$ is closed; hence $P^*(T)$ is compact. The continuity of $\rho_\alpha$ follows
from the fact that $g_\beta\circ\rho_\alpha=g_\beta$, for $\beta\le\alpha$, and $g_\beta\circ\rho_\alpha\equiv\infty$, for $\beta>\alpha$. Finally,
the topology induced on $\delta^{-1}(\alpha)$ by $\mathfrak p$ is equal to the topology induced by the maps $g_{\beta\alpha}$, with $\beta\in\delta[T]$
and $\beta\le\alpha$. The fact that such topology is equal to $\tau_\alpha$ follows from the continuity of the connecting maps and from the fact that $g_{\alpha\alpha}$ is the inclusion of $\delta^{-1}(\alpha)$ in $\delta^{-1}(\alpha)\cupdot\{\infty\}$.
\end{proof}


\begin{cor}\label{thm:limitapeso}
Let $(T,\delta)$ be a $\theta$-graded tree, $P^*(T)$ be endowed with a compatible topology and the sets $\delta^{-1}(\alpha)$ be endowed with
the topology induced by $\mathfrak p$. Then:
\[w\big(P^*(T)\big)\le\max\Big\{\vert\theta\vert,\sup_{\alpha\in\theta}w\big(\delta^{-1}(\alpha)\big)\Big\}.\]
In particular, if $\delta^{-1}(\alpha)$ is second countable for all $\alpha\in\omega_1$, then $P^*(T_\alpha)$ is second countable,
for all $\alpha\in\omega_1$.\qed
\end{cor}

\end{section}

\begin{section}{The counterexample to the conjecture}\label{sec:counterexample}

Combining Theorem~\ref{thm:propValdiviatree} with Proposition~\ref{thm:maquininha} and Corollary~\ref{thm:limitapeso}, we obtain the following strategy for constructing a small Valdivia compact space: take an $\omega_1$-graded tree $(T,\delta)$ and locally compact Hausdorff second countable topologies
on the sets $\delta^{-1}(\alpha)$ such that $\delta^{-1}(0)$ is compact and the connecting maps of $(T,\delta)$ are continuous. These topologies are then combined into a compatible topology on $P^*(T)$, which is a small Valdivia compact space. Moreover, every small Valdivia compact space is of this form.
The purpose of this section is to use this strategy to prove the following result.
\begin{teo}\label{thm:eba}
Assume $\diamondsuit$. There exists a Valdivia compact space $K$ such that:
\begin{itemize}
\item[(a)] $w(K)=\omega_1$;
\item[(b)] $K$ has ccc;
\item[(c)] $K$ has no $G_\delta$ points;
\item[(d)] $K$ does not have a nontrivial convergent sequence in the complement of a dense $\Sigma$-subset.
\end{itemize}
\end{teo}

\begin{rem}\label{thm:indepSigma}
We observe that property (d) in the statement of Theorem~\ref{thm:eba} is independent of the choice of the dense $\Sigma$-subset: more precisely, if $K$ is a Valdivia compact space and if $K\setminus S$ contains a nontrivial convergent sequence for some dense $\Sigma$-subset $S$ of $K$, then $K\setminus S$ contains a nontrivial convergent sequence for {\em any\/} dense $\Sigma$-subset $S$ of $K$ (see \cite[Remark~4.5]{ExtCKc0}).
\end{rem}

\smallskip

We start by investigating conditions on the $\omega_1$-graded tree $(T,\delta)$ and on the topologies of the sets $\delta^{-1}(\alpha)$ that imply
conditions (a)---(d) of Theorem~\ref{thm:eba}.
\begin{lem}\label{thm:PTGdeltapoint}
Let $(T,\delta)$ be an $\omega_1$-graded tree and let $P^*(T)$ be endowed with a compatible topology such that $P^*(T_\alpha)$ is metrizable, for all $\alpha\in\omega_1$. Then $P^*(T)$ has no $G_\delta$ points if and only if $\delta[N_A]$ is uncountable for every countable path $A\in P^*(T)$.
\end{lem}
\begin{proof}
Follows from Corollary~\ref{thm:KGdeltapoint} and Proposition~\ref{thm:KthetaTdelta}, keeping in mind that, for $\alpha\in\omega_1$ and $A\in P^*(T_\alpha)$, we have $\vert\rho_\alpha^{-1}(A)\vert=1$ if and only if $\delta[N_A]$ is contained in $[0,\alpha]$.
\end{proof}

In what follows, whenever $(T,\delta)$ is a graded tree and $P^*(T)$ is endowed with a compatible topology, we will consider the sets
$\delta^{-1}(\alpha)$ endowed with the topology induced by $\mathfrak p$.
\begin{lem}
Let $(T,\delta)$ be an $\omega_1$-graded tree and let $P^*(T)$ be endowed with a compatible topology such that $P^*(T_\alpha)$ is metrizable, for all $\alpha\in\omega_1$. Assume that \eqref{eq:sigmasubsettree} has empty interior in $P^*(T)$. Then $P^*(T)$ has ccc if and only if for every antichain $X\subset T$, the set:
\begin{equation}\label{eq:alphawithinterior}
\big\{\alpha\in\mathbb S:\text{$X\cap\delta^{-1}(\alpha)$ has nonempty interior in $\delta^{-1}(\alpha)$}\big\}
\end{equation}
is countable.
\end{lem}
\begin{proof}
By Corollary~\ref{thm:Kccc} and Proposition~\ref{thm:KthetaTdelta}, $P^*(T)$ does not have ccc if and only if there exist an uncountable subset $\Lambda$ of $\mathbb S\cap\omega_1$ and a family $(U_\alpha)_{\alpha\in\Lambda}$, with each $U_\alpha$ a nonempty open subset of $\delta^{-1}(\alpha)$, such that
the sets $\rho_\alpha^{-1}\big[\mathfrak p[U_\alpha]\big]$, $\alpha\in\Lambda$, are pairwise disjoint. The conclusion follows by observing
that these sets are pairwise disjoint if and only if $X=\bigcup_{\alpha\in\Lambda}U_\alpha$ is an antichain of $T$.
\end{proof}

\begin{cor}\label{thm:corccctree}
Let $(T,\delta)$ be an $\omega_1$-graded tree and let $P^*(T)$ be endowed with a compatible topology such that $P^*(T_\alpha)$ is metrizable, for all $\alpha\in\omega_1$. Assume that \eqref{eq:sigmasubsettree} has empty interior in $P^*(T)$. Set:
\begin{equation}\label{eq:D}
Z=\big\{t\in T:\text{$t$ is an isolated point of $\delta^{-1}\big(\delta(t)\big)$}\big\}.
\end{equation}
If $Z\cap\delta^{-1}(\alpha)$ is dense in $\delta^{-1}(\alpha)$, for all $\alpha\in\mathbb S\cup\{0\}$, then $P^*(T)$ has ccc if and only
if $Z$ has ccc.
\end{cor}
\begin{proof}
Note that the set \eqref{eq:alphawithinterior} is equal to $\delta[X\cap Z]\setminus\{0\}$ and that $Z\cap\delta^{-1}(\alpha)$ is countable,
for all $\alpha$.
\end{proof}

\begin{lem}\label{thm:uncountablepath}
Let $(T,\delta)$ be an $\omega_1$-graded tree and let $P^*(T)$ be endowed with a compatible topology such that $P^*(T_\alpha)$ is metrizable, for all $\alpha\in\omega_1$. Assume that $\delta^{-1}(\alpha)$ has at most one nonisolated point, for all $\alpha\in\mathbb S\cup\{0\}$, and that
the set $Z$ defined in \eqref{eq:D} is an initial part of $T$. If $P^*(T)$ has a nontrivial convergent sequence outside a dense $\Sigma$-subset, then
$Z$ contains an uncountable path.
\end{lem}
\begin{proof}
By Remark~\ref{thm:indepSigma}, we can assume that the complement of \eqref{eq:sigmasubsettree} has a nontrivial convergent sequence.
Let then $A\in P^*(T)$ be an uncountable path which is the limit of a sequence $(A_n)_{n\in\omega}$ in $P^*(T)\setminus\{A\}$ and let $\alpha_0\in\omega_1$
be such that $\rho_{\alpha_0}(A_n)\ne\rho_{\alpha_0}(A)$, for all $n\in\omega$. We claim that, for $\alpha\in\delta[A]$ with $\alpha\ge\alpha_0$, there
exists $n(\alpha)\in\omega$ such that $A_n\cap Z\cap\delta^{-1}(\alpha)\ne\emptyset$, for all $n\ge n(\alpha)$. Namely, for such $\alpha$
we have $g_\alpha(A)\ne\infty$, so that, by the continuity of $g_\alpha$, there exists $n(\alpha)\in\omega$ with $g_\alpha(A_n)\ne\infty$, for all $n\ge n(\alpha)$. Since $g_\alpha(A_n)\ne g_\alpha(A)$ for all $n$, we obtain that $g_\alpha(A)$ is not isolated in $\delta^{-1}(\alpha)$ and thus $g_\alpha(A_n)\in Z\cap\delta^{-1}(\alpha)$,
for all $n\ge n(\alpha)$. This proves the claim. To conclude the proof, pick $n\in\omega$ such that $n=n(\alpha)$ for uncountably many
$\alpha\in\delta[A]\setminus\alpha_0$ and note that $A_n$ is an uncountable path contained in $Z$.
\end{proof}

A tree $T$ is called {\em ever-branching\/} if, for every $t\in T$, the set
\[\big\{s\in T:s>t\big\}\]
is not a chain. We recall that if $T$ is an ever-branching tree
with ccc such that $\left]\cdot,t\right[$ is countable, for all $t\in T$, then every path of $T$ is countable (see \cite[Lemma~7.4]{Kunen}).

\begin{lem}\label{thm:chega}
Let $(T,\delta)$ be a nonempty $\omega_1$-graded tree and let $P^*(T)$ be endowed with a compatible topology such that $P^*(T_\alpha)$ is metrizable, for all $\alpha\in\omega_1$. Assume that:
\begin{itemize}
\item[(i)] $N_A$ is uncountable, for every countable path $A\in P^*(T)$;
\item[(ii)] $\delta^{-1}(\alpha)$ has at most one nonisolated point, for all $\alpha\in\mathbb S\cup\{0\}$;
\item[(iii)] the set $Z$ defined in \eqref{eq:D} is an initial part of $T$ with ccc.
\end{itemize}
Then $K=P^*(T)$ is a Valdivia compact space satisfying conditions (a)---(d) in the statement of Theorem~\ref{thm:eba}.
\end{lem}
\begin{proof}
By Theorem~\ref{thm:propValdiviatree}, $K$ is a Valdivia compact space, \eqref{eq:sigmasubsettree} is a dense $\Sigma$-subset of $K$ and $w(K)\le\omega_1$. Since $\delta^{-1}(\alpha)$ is countable, for all $\alpha\in\mathbb S\cup\{0\}$, we have that $\delta[N_A]$ is uncountable, for every
countable path $A\in P^*(T)$, so that Lemma~\ref{thm:PTGdeltapoint} implies that $K$ has no $G_\delta$ points. It then follows from Lemma~\ref{thm:ValdiviaGdeltaSigma} that \eqref{eq:sigmasubsettree} has empty interior and from Corollary~\ref{thm:corccctree} that $K$ has ccc. By Lemma~\ref{thm:uncountablepath}, to conclude the proof, it suffices to show that $Z$ does not contain an uncountable path. To this aim, we show that $Z$ must be ever-branching. Let $z\in Z$
and note that, since $N_{\mathfrak p(z)}$ is an uncountable antichain, there exists $t\in T\setminus Z$ with $t\in N_{\mathfrak p(z)}$. Setting $\alpha=\delta(z)$ and $\beta=\delta(t)$, we have that $g_{\alpha\beta}(t)=z$. Since $z$ is isolated in $\delta^{-1}(\alpha)$ and $t$ is not isolated in $\delta^{-1}(\beta)$, it follows from the continuity of $g_{\alpha\beta}$ that $g_{\alpha\beta}^{-1}(z)$ is infinite; moreover,
by (ii), $g_{\alpha\beta}^{-1}(z)\setminus\{t\}$ is contained in $Z$. Hence, $g_{\alpha\beta}^{-1}(z)\setminus\{t\}$ is an infinite
antichain in $Z$ consisting of elements greater than $z$.
\end{proof}

Our goal now is to construct an $\omega_1$-graded tree $(T,\delta)$ and a compatible topology on $P^*(T)$ such that the assumptions of Lemma~\ref{thm:chega}
hold. Observe that the initial subtree $Z$ of $T$ will be a {\em Suslin tree}, i.e., $\vert Z\vert=\omega_1$, every path of $Z$ is countable and $Z$ has ccc.
Our construction is similar to the standard construction of a Suslin tree using $\diamondsuit$ (\cite[Theorem~7.8]{Kunen}), but technically much more involved, since we have to ensure the continuity of the connecting maps. The first step is to develop the technique that will be later used to prove that $Z$ has ccc.
The construction itself is the purpose of Subsection~\ref{sub:construction}.

\begin{defin}
Let $(T,\delta)$ be a graded tree. An antichain $X$ of $T$ is said to be {\em special\/} in $T$ if given a finite subset $\mathcal F$
of $\mathbb S\cup\{0\}$ and an element $t\in T$ with $\mathfrak p(t)\cap X=\emptyset$, there exists $s\in X$ with $s>t$ such that
$\delta\big[\left]t,s\right]\big]\cap\mathcal F=\emptyset$, where $\left]t,s\right]=\big\{u\in T:t<u\le s\big\}$.
\end{defin}
Obviously a special antichain is also a maximal antichain.

\begin{lem}\label{thm:lemmaclub}
Let $(T,\delta)$ be an $\omega_1$-graded tree and let $X$ be a special antichain in $T$. If $T_\alpha$ is countable for all $\alpha\in\omega_1$,
then the set:
\begin{equation}\label{eq:isclub}
\big\{\alpha\in\omega_1:\text{$X\cap T_\alpha$ is a special antichain in $T_\alpha$}\big\}
\end{equation}
is closed and unbounded (club).
\end{lem}
\begin{proof}
The set is clearly closed. To see that it is unbounded note that, given $\alpha\in\omega_1$, the fact that $T_\alpha$ is countable implies that there exists $f(\alpha)$ in $\omega_1$ with $f(\alpha)>\alpha$
having the following property: for every $t\in T_\alpha$ with $\mathfrak p(t)\cap X=\emptyset$ and every finite subset $\mathcal F$ of $[0,\alpha]$,
there exists $s\in X\cap T_{f(\alpha)}$ with $s>t$ and $\delta\big[\left]t,s\right]\big]\cap\mathcal F=\emptyset$. Hence, denoting by $f^n$ the $n$-th iterate of $f$,
we have that $\sup_{n\in\omega}f^n(\alpha)$ is in \eqref{eq:isclub}, for all $\alpha\in\omega_1$.
\end{proof}

Given a set $\Gamma$ with $\vert\Gamma\vert=\omega_1$, by a {\em continuous filtration of $\Gamma$ by countable sets\/} we mean an increasing family $(\Gamma_\alpha)_{\alpha\in\omega_1}$ of countable subsets of $\Gamma$ such that $\Gamma=\bigcup_{\alpha\in\omega_1}\Gamma_\alpha$ and $\Gamma_\alpha=\bigcup_{\beta\in\alpha}\Gamma_\beta$, for every limit ordinal $\alpha\in\omega_1$.
A {\em $\diamondsuit$-sequence\/} for this filtration is a family $(\Gamma^\diamondsuit_\alpha)_{\alpha\in\omega_1}$ such that each $\Gamma^\diamondsuit_\alpha$ is a subset of $\Gamma_\alpha$ and, for any $X\subset\Gamma$, the set $\big\{\alpha\in\omega_1:X\cap\Gamma_\alpha=\Gamma^\diamondsuit_\alpha\big\}$ is stationary. The combinatorial principle $\diamondsuit$ states that
there exists a $\diamondsuit$-sequence for the filtration $\omega_1=\bigcup_{\alpha\in\omega_1}\alpha$. It is easy to prove that $\diamondsuit$ implies
the existence of a $\diamondsuit$-sequence for any continuous filtration by countable sets $\Gamma=\bigcup_{\alpha\in\omega_1}\Gamma_\alpha$.

\begin{cor}\label{thm:temccc}
Let $(T,\delta)$ be an $\omega_1$-graded tree such that $T_\alpha$ is countable, for all $\alpha\in\omega_1$. Let $(T^\diamondsuit_\alpha)_{\alpha\in\omega_1}$ be a $\diamondsuit$-sequence for the continuous filtration $T=\bigcup_{\alpha\in\omega_1}T_\alpha$. Assume that:
\begin{itemize}
\item[(i)] for all $\alpha\in\omega_1$, if $T^\diamondsuit_\alpha$ is a special antichain in $T_\alpha$, then $\mathfrak p(t)$ intersects $T^\diamondsuit_\alpha$, for every $t\in T\setminus T_\alpha$;
\item[(ii)] every maximal antichain in $T$ is special in $T$.
\end{itemize}
Then $T$ has ccc.
\end{cor}
\begin{proof}
Given a maximal antichain $X$ in $T$, from (ii) and Lemma~\ref{thm:lemmaclub} we obtain $\alpha\in\omega_1$ such that $X\cap T_\alpha$ is a special antichain in $T_\alpha$ and $X\cap T_\alpha=T^\diamondsuit_\alpha$. It then follows from (i) that $X\subset T_\alpha$.
\end{proof}

\subsection{Construction of the graded tree}\label{sub:construction} Throughout this subsection, $T$ and $Z$ are defined by:
\[T=[0,\omega]\times\big[\big(\mathbb S\cup\{0\}\big)\cap\omega_1\big],\quad Z=\omega\times\big[\big(\mathbb S\cup\{0\}\big)\cap\omega_1\big],\]
and $\delta:T\to\mathbb S\cup\{0\}$ denotes the projection onto the second coordinate. As usual, we write
$T_\alpha=\delta^{-1}\big[[0,\alpha]\big]$ and $Z_\alpha=Z\cap T_\alpha$. Moreover, for $\alpha$ in $\big(\mathbb S\cup\{0\}\big)\cap\omega_1$,
we endow $\delta^{-1}(\alpha)=[0,\omega]\times\{\alpha\}$ with the topology that makes the projection onto the first coordinate a homeomorphism,
where $[0,\omega]$ has the order topology. Clearly, equality \eqref{eq:D} holds. The hard part of the work will be the construction of the partial order of $T$.

\begin{defin}\label{thm:defadmissible}
Let $(Z^\diamondsuit_\alpha)_{\alpha\in\omega_1}$ be a $\diamondsuit$-sequence for the continuous filtration $Z=\bigcup_{\alpha\in\omega_1}Z_\alpha$.
Given $\alpha\in[0,\omega_1]$, we say that a partial order $\le$ in $T_\alpha$ is {\em admissible\/} (with respect to the given $\diamondsuit$-sequence) if it satisfies the following properties:
\begin{enumerate}
\item $(T_\alpha,{\le})$ is a tree, $\delta\vert_{T_\alpha}$ is a grading function and the connecting maps of $(T_\alpha,\delta\vert_{T_\alpha})$
are continuous;
\item $Z_\alpha$ is an initial part of $(T_\alpha,{\le})$;
\item for all $\beta<\alpha$, if $Z^\diamondsuit_\beta$ is a special antichain in $(Z_\beta,\delta\vert_{Z_\beta})$, then
$\mathfrak p(z)$ intersects $Z^\diamondsuit_\beta$, for every $z\in Z_\alpha\setminus Z_\beta$.
\end{enumerate}
\end{defin}
Note that (1) and (2) imply that $T_\beta$ and $Z_\beta$ are initial parts of $(T_\alpha,{\le})$, for all $\beta\le\alpha$; in particular,
$(T_\beta,\delta\vert_{T_\beta})$ and $(Z_\beta,\delta\vert_{Z_\beta})$ are graded trees (endowed with the restriction of $\le$).

\smallskip

In Corollary~\ref{thm:corsofaltaesse} below, we show that if $T$ is endowed with an admissible partial order satisfying a simple extra property,
then the assumptions of Lemma~\ref{thm:chega} are satisfied.
\begin{lem}
Let $\le$ be an admissible partial order in $T$ and assume that, for every $z\in Z$, the set $N_{\mathfrak p(z)}\setminus Z$ is uncountable.
Then every maximal antichain in $Z$ is special in $Z$.
\end{lem}
\begin{proof}
Let $X$ be a maximal antichain in $Z$ and pick $z\in Z$ with $\mathfrak p(z)\cap X=\emptyset$; write $\alpha=\delta(z)$. If $\mathcal F$ is a finite subset of $\big(\mathbb S\cup\{0\}\big)\cap\omega_1$, then there exists $t\in N_{\mathfrak p(z)}\setminus Z$ with $\delta(t)=\beta>\sup\mathcal F$. Since $g_{\alpha\beta}(t)=z$ and $g_{\gamma\beta}(t)=\infty$, for all $\gamma\in\left]\alpha,\beta\right[\cap\mathbb S$, it follows
from the continuity of the connecting maps that there exists $w\in Z\cap\delta^{-1}(\beta)$ such that $g_{\alpha\beta}(w)=z$ and
$g_{\gamma\beta}(w)=\infty$, for all $\gamma\in\mathcal F$ with $\gamma>\alpha$. Note that $w>z$ and $\delta\big[\left]z,w\right]\big]\cap\mathcal F=\emptyset$. To conclude the proof, use the fact that $X$ is a maximal antichain in $Z$ to obtain $v\in X$ comparable with $w$ and note that
$v>z$ and $\delta\big[\left]z,v\right]\big]\cap\mathcal F=\emptyset$.
\end{proof}

\begin{cor}\label{thm:corsofaltaesse}
Let $\le$ be an admissible partial order in $T$ and assume that, for every countable path $A\in P^*(T)$, the set $N_A\setminus Z$ is uncountable.
If $P^*(T)$ is endowed with the compatible topology given by Proposition~\ref{thm:maquininha}, then $K=P^*(T)$ is a Valdivia compact space satisfying conditions (a)---(d) in the statement of Theorem~\ref{thm:eba}.
\end{cor}
\begin{proof}
It follows from Corollary~\ref{thm:temccc} that $Z$ has ccc and from Corollary~\ref{thm:limitapeso} that $P^*(T_\alpha)$ is metrizable,
for all $\alpha\in\omega_1$. The assumptions of Lemma~\ref{thm:chega} are thus satisfied.
\end{proof}

We are going to construct, by recursion, an admissible partial order in $T$ satisfying the assumption of Corollary~\ref{thm:corsofaltaesse}. Note that, if $\alpha\in[0,\omega_1]$ is a limit ordinal and, for each $\beta<\alpha$, an admissible partial order ${\le}_\beta$ is given in $T_\beta$
such that $(T_\beta,{\le}_\beta)$ is a subtree of $(T_\gamma,{\le}_\gamma)$, for all $\beta\le\gamma<\alpha$, then ${\le}_\alpha=\bigcup_{\beta<\alpha}{\le}_\beta$ is an admissible partial order in $T_\alpha$. For the recursion step, we will use the lemma below.

\begin{lem}\label{thm:estendeadmissible}
Given $\alpha\in\omega_1$, an admissible partial order $\le$ in $T_\alpha$ and a path $A\in P^*(T_\alpha)$, there exists an admissible
partial order ${\le}'$ in $T_{\alpha+1}$ such that $(T_\alpha,{\le})$ is a subtree of $(T_{\alpha+1},{\le}')$ and
$A=\big\{t\in T_\alpha:t<'(\omega,\alpha+1)\big\}$.
\end{lem}

We postpone for a moment the proof of Lemma~\ref{thm:estendeadmissible} to conclude the proof of Theorem~\ref{thm:eba}.

\begin{cor}\label{thm:psi}
There exists an admissible partial order $\le$ in $T$ such that $N_A\setminus Z$ is uncountable, for every countable path $A\in P^*(T)$.
\end{cor}
\begin{proof}
Denote by $\wp_\omega(T)$ the collection of all countable subsets of $T$ and let $\psi:\omega_1\to\wp_\omega(T)$ be a map with $\psi^{-1}(A)$ uncountable,
for all $A\in\wp_\omega(T)$. Using Lemma~\ref{thm:estendeadmissible}, construct by recursion a family $({\le}_\alpha)_{\alpha\in\omega_1}$,
with ${\le}_\alpha$ an admissible partial order in $T_\alpha$, such that $(T_\alpha,{\le}_\alpha)$ is a subtree of $(T_\beta,{\le}_\beta)$,
for $\alpha\le\beta\in\omega_1$, and such that, for all $\alpha\in\omega_1$, we have
\[\big\{t\in T_\alpha:t<_{\alpha+1}(\omega,\alpha+1)\big\}=A,\]
where $A=\psi(\alpha)$, if $\psi(\alpha)\in P^*(T_\alpha)$, and $A=\{(\omega,0)\}$, otherwise. To conclude the proof,
set ${\le}=\bigcup_{\alpha\in\omega_1}{\le}_\alpha$.
\end{proof}

\begin{proof}[Proof of Theorem~\ref{thm:eba}]
Follows directly from Corollaries~\ref{thm:corsofaltaesse} and \ref{thm:psi}.
\end{proof}

We now turn to the proof of Lemma~\ref{thm:estendeadmissible}. In order to extend an admissible partial order from $T_\alpha$ to $T_{\alpha+1}$,
we need to specify, for each $n\in[0,\omega]$, a path $A_n\in P^*(T_\alpha)$ which will be the set of predecessors of $(n,\alpha+1)$.
We then give the following definition.
\begin{defin}
Given $\alpha\in\omega_1$, let $\le$ be a partial order in $T_\alpha$ such that $(T_\alpha,{\le})$ is a tree and let $(A_n)_{n\in[0,\omega]}$ be a sequence in $P(T_\alpha)$. We define a partial order ${\le}'$ in $T_{\alpha+1}$ by requiring that, for all $t,s\in T_{\alpha+1}$, we have $t<'s$ if and only if one of the following conditions holds:
\begin{itemize}
\item $t,s\in T_\alpha$ and $t<s$;
\item $t\in A_n$ and $s=(n,\alpha+1)$, for some $n\in[0,\omega]$.
\end{itemize}
We call ${\le}'$ the partial order in $T_{\alpha+1}$ {\em induced\/} by the sequence $(A_n)_{n\in[0,\omega]}$.
\end{defin}

\begin{lem}\label{thm:lema123}
Let $\alpha\in\omega_1$ and fix a partial order $\le$ in $T_\alpha$ such that $(T_\alpha,\le)$ is a tree and $\delta\vert_{T_\alpha}$ is a grading function.
Let $(A_n)_{n\in[0,\omega]}$ be a sequence in $P^*(T_\alpha)$ and ${\le}'$ be the partial order induced in $T_{\alpha+1}$. The following statements hold:
\begin{enumerate}
\item $(T_{\alpha+1},{\le}')$ is a tree, $\delta\vert_{T_{\alpha+1}}$ is a grading function and $(T_\alpha,{\le})$ is a subtree of $(T_{\alpha+1},{\le}')$;
\item if $Z_\alpha$ is an initial part of $(T_\alpha,{\le})$ and $A_n$ is in $P^*(Z_\alpha)$, for all $n\in\omega$, then $Z_{\alpha+1}$ is an initial
part of $(T_{\alpha+1},{\le}')$.
\item Assume that the connecting maps of $(T_\alpha,\delta\vert_{T_\alpha})$ are continuous and let $\beta\le\alpha$ in $\mathbb S\cup\{0\}$ be such that the connecting map $g_{\beta,\alpha+1}$ of $(T_{\alpha+1},\delta\vert_{T_{\alpha+1}})$ is continuous. If $g_{\beta,\alpha+1}(\omega,\alpha+1)\ne\infty$, then the connecting map $g_{\gamma,\alpha+1}$ of $(T_{\alpha+1},\delta\vert_{T_{\alpha+1}})$ is continuous, for all $\gamma\le\beta$ in $\mathbb S\cup\{0\}$.
\end{enumerate}
\end{lem}
\begin{proof}
The proofs of (1) and (2) are straightforward and, to prove (3), note that the equality $g_{\gamma,\alpha+1}=g_{\gamma\beta}\circ g_{\beta,\alpha+1}$ holds in the open subset
\[\big\{t\in\delta^{-1}(\alpha+1):g_{\beta,\alpha+1}(t)\ne\infty\big\}\]
of $\delta^{-1}(\alpha+1)$.
\end{proof}

Our task now is to find an appropriate sequence $(A_n)_{n\in[0,\omega]}$ to induce the partial order of $T_{\alpha+1}$.
\begin{lem}\label{thm:zn}
Let $\alpha\in\omega_1$ and fix a partial order $\le$ in $T_\alpha$ satisfying conditions (1) and (2) of Definition~\ref{thm:defadmissible}. Given $A_\omega\in P^*(T_\alpha)$, there exists a sequence $(z_n)_{n\in\omega}$ in $Z_\alpha$ satisfying the following properties:
\begin{itemize}
\item[(a)] $\sup_{n\in\omega}\delta(z_n)=\sup\delta[A_\omega]$;
\item[(b)] given a sequence $(A_n)_{n\in\omega}$ in $P^*(T_\alpha)$ with $z_n\in A_n$, for all $n\in\omega$, if $T_{\alpha+1}$ is endowed with the partial order induced by $(A_n)_{n\in[0,\omega]}$, then the connecting maps $g_{\beta,\alpha+1}$ of the graded tree $T_{\alpha+1}$ are continuous,
    for all $\beta\in\mathbb S\cup\{0\}$ with $\beta\le\sup\delta[A_\omega]$.
\end{itemize}
\end{lem}
\begin{proof}
Assume first that $A_\omega\subset Z_\alpha$. If $A_\omega$ has a largest element $z$, take $z_n=z$, for all $n\in\omega$; otherwise,
let $(z_n)_{n\in\omega}$ be an increasing cofinal sequence in $A_\omega$. Assume now that $A_\omega$ is not contained in $Z_\alpha$.
If $A_\omega$ has a largest element $t$, then it must be in $T_\alpha\setminus Z_\alpha$, since $Z_\alpha$ is an initial part of $T_\alpha$.
Then $t$ is of the form $(\omega,\beta)$ and we take $z_n=(n,\beta)$, for all $n\in\omega$. Finally, if $A_\omega$ does not have a largest element,
let $\big((\omega,\beta_n)\big)_{n\in\omega}$ be an increasing cofinal sequence in $A_\omega$. For each $n\in\omega$,
since $g_{\beta_i\beta_n}$ is continuous and $g_{\beta_i\beta_n}(\omega,\beta_n)=(\omega,\beta_i)$, for $i\le n$, we can take $z_n\in\omega\times\{\beta_n\}$ with $g_{\beta_i\beta_n}(z_n)\in[n,\omega]\times\{\beta_i\}$, for all $i\le n$. In all four cases considered above, it is easy to check that
$(z_n)_{n\in\omega}$ satisfies properties (a) and (b), keeping in mind item~(3) of Lemma~\ref{thm:lema123}.
\end{proof}

\begin{lem}\label{thm:B}
Let $\alpha\in\omega_1$ and fix an admissible partial order $\le$ in $T_\alpha$. Given $z\in Z_\alpha$ and a finite subset $\mathcal F$ of $\mathbb S\cup\{0\}$,
there exists $B\in P^*(Z_\alpha)$ with $z\in B$ satisfying the following conditions:
\begin{itemize}
\item[(a)] for all $\beta\le\alpha$, if $Z^\diamondsuit_\beta$ is a special antichain in $Z_\beta$, then
$B$ intersects $Z^\diamondsuit_\beta$;
\item[(b)] for all $w\in B$ with $w>z$, we have $\delta(w)\not\in\mathcal F$.
\end{itemize}
\end{lem}
\begin{proof}
Note first that if $\beta\le\alpha$ and $Z^\diamondsuit_\beta$ is a special antichain in $Z_\beta$, then $Z^\diamondsuit_\beta$ is also a special
antichain in $Z_\alpha$, by property~(3) in Definition~\ref{thm:defadmissible}. Set:
\[\Lambda=\big\{\beta\in[0,\alpha]:\text{$Z^\diamondsuit_\beta$ is a special antichain in $Z_\beta$ and $\mathfrak p(z)\cap Z^\diamondsuit_\beta=\emptyset$}\big\}.\]
If $\Lambda=\emptyset$, take $B=\mathfrak p(z)$. Otherwise, let $\{\beta_n:n\in\omega\}$ be an enumeration of $\Lambda$ and let us define by recursion
a chain $\{w_n:n\in\omega\}$ such that $w_n\in Z^\diamondsuit_{\beta_n}$, $w_n>z$, and $\delta\big[\left]z,w_n\right]\big]\cap\mathcal F=\emptyset$, for all $n\in\omega$. Given $w_i$, $i\le n$, we obtain $w_{n+1}$ as follows. Setting $w=\max\{w_0,\ldots,w_n\}$, take $w_{n+1}\in\mathfrak p(w)\cap Z^\diamondsuit_{\beta_{n+1}}$, if this intersection is not empty; otherwise, take $w_{n+1}\in Z^\diamondsuit_{\beta_{n+1}}$ with $w_{n+1}>w$ and
$\delta\big[\left]w,w_{n+1}\right]\big]\cap\mathcal F=\emptyset$. To conclude the proof, set $B=\bigcup_{n\in\omega}\mathfrak p(w_n)$.
\end{proof}

\begin{proof}[Proof of Lemma~\ref{thm:estendeadmissible}]
Set $A_\omega=A$ and take $(z_n)_{n\in\omega}$ in $Z_\alpha$ as in Lemma~\ref{thm:zn}. By recursion, we define $A_n\in P^*(Z_\alpha)$ and
injective maps $\phi_n:A_n\to\omega$ as follows: given $A_i$ and $\phi_i$, for $i<n$, obtain $B\in P^*(Z_\alpha)$ from Lemma~\ref{thm:B} with $z=z_n$ and
\begin{equation}\label{eq:F}
\mathcal F=\bigcup_{i<n}\delta\Big[\phi_i^{-1}\big[[0,n]\big]\Big];
\end{equation}
set $A_n=B$. Let ${\le}'$ be the partial order in $T_{\alpha+1}$ induced by the sequence $(A_n)_{n\in[0,\omega]}$. That
${\le}'$ satisfies property~(3) of Definition~\ref{thm:defadmissible} follows from the fact that each $A_n$ satisfies condition~(a) in Lemma~\ref{thm:B}.
To conclude the proof, fix $\beta\in\mathbb S\cup\{0\}$ with $\sup\delta[A_\omega]<\beta\le\alpha$ and let us check that $g_{\beta,\alpha+1}$ is continuous at the point $(\omega,\alpha+1)$. To this aim, it is sufficient to verify that $\beta\in\delta[A_n]$ for at most a finite number of $n\in\omega$.
If $\beta=\delta(v)$, for some $v\in A_i$ with $i\in\omega$, we claim that $\beta\not\in\delta[A_n]$, for all $n>\max\big\{i,\phi_i(v)\big\}$.
Namely, for such $n$, we have that $\beta$ is in \eqref{eq:F} and then, by the construction of $A_n$, we get that
$\delta(w)\ne\beta$, for all $w\in A_n$ with $w>z_n$. Finally,
for $w\le z_n$, we have $\delta(w)\le\delta(z_n)\le\sup\delta[A_\omega]<\beta$, proving the claim.
\end{proof}

\end{section}

\end{document}